\documentclass{amsart}
\usepackage[utf8]{inputenc}

\usepackage{amsmath}
\usepackage{amsthm}
\usepackage{amssymb}
\usepackage[all]{xy}
\usepackage{mathrsfs}

\newcommand{\QQ}{\mathbb{Q}}

\newcommand{\kk}{\Bbbk}

\newcommand{\MC}{\mathsf{MC}}
\newcommand{\MCQ}{\widetilde{\mathsf{MC}}}
\newcommand{\gauge}{\mathsf{G}}
\renewcommand{\exp}{\mathsf{exp}}

\newcommand{\glike}{\mathcal{G}}
\newcommand{\prim}{\mathcal{P}}
\newcommand{\Kangrp}{G}

\newcommand{\sgn}{\operatorname{sgn}}

\newcommand{\DGL}{\mathsf{DGL}}
\newcommand{\Top}{\mathsf{Top}}

\newcommand{\cU}{\widehat{U}}
\newcommand{\cQQ}{\widehat{\QQ}}

\newcommand{\horb}{\mathbin{/\mkern-6mu/}}

\newcommand{\ad}{\operatorname{ad}}
\newcommand{\tensor}{\otimes}
\newcommand{\ctensor}{\widehat{\otimes}}

\newcommand{\CC}{\mathcal{C}}

\newcommand{\vol}[1]{\omega^{#1}}
\newcommand{\volo}[1]{\widehat{\omega}^{#1}}

\newtheorem{theorem}{Theorem}[section]

\newtheorem{proposition}[theorem]{Proposition}
\newtheorem{corollary}[theorem]{Corollary}
\newtheorem{lemma}[theorem]{Lemma}

\theoremstyle{definition}
\newtheorem{definition}[theorem]{Definition}
\newtheorem{remark}[theorem]{Remark}

\title{On exponential groups and Maurer--Cartan spaces}
\author{Alexander Berglund}

\begin{document}

\maketitle

\begin{abstract}
The purpose of this note is to give a concise account of some fundamental properties of the exponential group and the Maurer--Cartan space associated to a complete dg Lie algebra. In particular, we give a direct elementary proof that the Maurer--Cartan space is a delooping of the exponential group. This leads to a short proof that the Maurer--Cartan space functor is homotopy inverse to Quillen's functor from simply connected pointed spaces to positively graded dg Lie algebras.
\end{abstract}

\section{Introduction}
The Maurer--Cartan space, or nerve, $\MC_\bullet(L)$ of a dg Lie algebra $L$, introduced in the context of deformation theory by Hinich \cite{Hinich} and studied extensively by Getzler \cite{Getzler}, has come to play a significant role in rational homotopy theory (see e.g.~\cite{B, LM, BFMT, RNV}).

The main purpose of this note is to record an elementary proof that the Maurer--Cartan space $\MC_\bullet(L)$ is a delooping of the exponential group $\exp_\bullet(L)$ for arbitrary complete dg Lie algebras $L$. As a corollary, we observe that this leads to a quick proof that the Maurer--Cartan space functor is inverse to Quillen's equivalence $\lambda$ from the rational homotopy category of simply connected pointed spaces to the homotopy category of positively graded dg Lie algebras \cite{Quillen}. This recovers the main result of \cite{FFM} and leads to a new, more direct, proof of the Baues--Lemaire conjecture \cite[Conjecture 3.5]{BL}, first solved by Majewski \cite{Majewski}.

Along the way, we give streamlined proofs of other fundamental properties, e.g., that $\MC_\bullet(-)$ takes surjections to Kan fibrations, that gauge equivalence agrees with homotopy equivalence, that $\MC_\bullet(L)$ is equivalent to the nerve of the Deligne groupoid for non-positively graded $L$, and we give an efficient computation of the homotopy groups together with Whitehead products and the action of the fundamental group.

\section{Complete dg Lie algebras}
Let $\kk$ be a field of characteristic zero. By a \emph{cochain complex} we mean an unbounded cochain complex $V$ of $\kk$-vector spaces,
$$\cdots \to V^{-1} \to V^0 \to V^1 \to \cdots.$$
We use the standard convention $V_k = V^{-k}$ to view a cochain complex as a chain complex when convenient. A \emph{complete filtered cochain complex} is a cochain complex $V$ equipped with a filtration of subcomplexes
$$V = F^1V \supseteq F^2V \supseteq \ldots$$
such that the canonical map $V \to \varprojlim V/F^n V$ is an isomorphism.

A \emph{complete dg Lie algebra} is a complete filtered cochain complex $L$ together with a Lie bracket $L\tensor L \to L$ that is compatible with the filtration in the sense that $[F^pL,F^qL] \subseteq F^{p+q}L$ for all $p,q$. Nilpotent dg Lie algebras equipped with the lower central series filtration are basic examples of complete dg Lie algebras.

Let $\MC(L)$ be the set of elements $\tau\in L^1$ that satisfy the Maurer-Cartan equation
$$d\tau + \tfrac{1}{2}[\tau,\tau] = 0.$$
Setting $d_\tau = d + \ad_\tau$, where $\ad_\tau(x) = [\tau,x]$, this equation is equivalent to
$$d_\tau^2 = 0.$$
We let $L_\tau$ be the dg Lie algebra with underlying Lie algebra $L$ and differential $d_\tau$.

Let $\gauge(L)$ denote the group $\exp(L^0)$, i.e., the group with underlying set $L^0$ and multiplication given by the Baker--Campbell--Hausdorff formula. The \emph{gauge action} of $x\in\gauge(L)$ on $\tau\in\MC(L)$ is given by the well-known formula
\begin{equation*}
x\cdot \tau = \tau - \sum_{k\geq 0} \tfrac{1}{(k+1)!}\ad_x^{k}(d_\tau x).
\end{equation*}
A direct verification that this defines a group action can be found in e.g.~\cite[\S4.3]{BFMT}.

\begin{lemma}\label{lemma:stabilizer}
The stabilizer $G(L)_\tau$ of a Maurer--Cartan element $\tau$ in $L$ under the gauge action may be identified with the subgroup
$$\exp(Z^0L_\tau) \subseteq \exp(L^0).$$
\end{lemma}

\begin{proof}
If $d_\tau x = 0$, then clearly $x\cdot \tau = \tau$.  Conversely, if $x\cdot \tau = \tau$, then
$$d_\tau x = - \tfrac{1}{2!}[x,d_\tau x] - \tfrac{1}{3!}[x,[x,d_\tau x ]] - \ldots,$$
which shows that $d_\tau x\in F^kL$ implies $d_\tau x \in F^{k+1}L$. Hence, $d_\tau x \in F^kL$ for all $k$ by induction, so $d_\tau x = 0$ by completeness.
\end{proof}
We denote the set of gauge equivalence classes of Maurer--Cartan elements by
$$\MCQ(L) = \MC(L)/\gauge(L).$$
The only non-trivial fact we will need about $\MCQ(L)$ is the following light version of the Goldman--Millson theorem \cite{GM}. We give a short proof for completeness.

\begin{lemma} \label{lemma:GM light}
If $f\colon L\to L'$ is a surjective morphism of complete dg Lie algebras such that $H^1(F^nI/F^{n+1}I) = 0$, where $I$ is the kernel of $f$ and $F^nI = F^nL\cap I$, then the map $\MCQ(L)\to \MCQ(L')$ induced by $f$ is injective.
\end{lemma}

\begin{proof}
Suppose $\tau,\rho \in \MC(L)$ and $f(\tau) = y \cdot f(\rho)$ for $y\in (L')^0$. Pick $x_1\in L^0$ with $f(x_1) = y$ and assume by induction that we have found $x_1,\ldots,x_n\in L^0$ satisfying
\begin{enumerate}
\item $x_k\equiv x_{k-1}  \pmod{F^{k-1}I}$,
\item $\tau \equiv x_k\cdot \rho \pmod{F^kI}$,
\end{enumerate}
for all $k\leq n$. Then (2) for $k=n$ implies that
$$d\tau = -\tfrac{1}{2}[\tau,\tau] \equiv
-\tfrac{1}{2}[x_n\cdot \rho , x_n\cdot \rho]  = d(x_n\cdot\rho) \quad \pmod{F^{n+1}I}.$$
Thus, $\tau-x_n\cdot \rho$ represents a cocycle in $F^nI/F^{n+1}I$. Since $H^1(F^nI/F^{n+1}I) = 0$, we can find $c\in F^nI$ such that
$$\tau \equiv  x_n\cdot \rho - dc \equiv (x_n+c)\cdot \rho \pmod{F^{n+1}I},$$
where the last congruence follows since $c\in F^nI$. Setting $x_{n+1} = x_n + c$ finishes the induction. By completeness, this defines an element $x\in L^0$ such that $\tau = x\cdot \rho$.
\end{proof}

\section{Polynomial differential forms}
Let $\Omega_\bullet$ be the simplicial commutative dg algebra of polynomial differential forms. In level $n$, it is given by
$$\Omega_n = \kk\big[t_0,\ldots,t_n,dt_0,\ldots,dt_n\big]/I$$
where $t_i$ has degree 0 and $I$ is the dg ideal generated by $t_0+\ldots+t_n-1$. The face maps $\partial_i\colon \Omega_n\to \Omega_{n-1}$ and the degeneracy maps $s_i\colon \Omega_n \to \Omega_{n+1}$ are given by
\begin{align*}
\partial_i(\omega)(t_0,\ldots,t_{n-1}) & = \omega(t_0,\ldots,t_{i-1},0,t_i,\ldots,t_{n-1}), \\
s_i(\eta)(t_0,\ldots,t_{n+1}) & = \eta(t_0,\ldots,t_i+t_{i+1},\ldots,t_{n+1}),
\end{align*}
for $0\leq i \leq n$ (cf.~e.g.~\cite[\S10(c)]{FHT-RHT}). Here are two key properties of $\Omega_\bullet$.
\begin{lemma} \label{lemma:poincare lemma}
The unit map $\kk \to \Omega_n$ is a quasi-isomorphism for every $n$.
\end{lemma}
\begin{proof}
The unit map of $\Omega_1 \cong \kk[t,dt]$ is easily seen to be a quasi-isomorphism if and only if $\kk$ has characteristic zero. Since $\Omega_n$ is isomorphic to $\Omega_1^{\tensor n}$, the claim for $n\ne 1$ follows from the K\"unneth theorem.
\end{proof}
\begin{lemma} \label{lemma:contractibility}
The simplicial vector space $\Omega_\bullet^k$ is contractible for every $k$.
\end{lemma}

\begin{proof}
Classes in $\pi_n(\Omega_\bullet^k)$ are represented by differential forms $\omega \in \Omega_n^k$ that satisfy $\partial_i(\omega) = 0$ for $0\leq i\leq n$. Given such a form $\omega$, the formula
$$
\nu = \sum_{j=1}^{n +1} t_j \omega(t_1,\ldots,t_{j-1},t_j+t_0,t_{j+1},\ldots,t_{n+1})$$
defines a form $\nu \in \Omega_{n + 1}^k$ such that $\partial_0(\nu) = \omega$ and $\partial_i(\nu) = 0$ for $0<i\leq n$, showing the homotopy class represented by $\omega$ is trivial.
\end{proof}

\begin{remark}
As discussed in \cite[\S7.3-\S7.4]{B-rathom}, contractibility of the simplicial cdga $\Omega_\bullet$ is equivalent to it being \emph{extendable} in the sense of \cite[p.118]{FHT-RHT}. Extendability is the basic property that the restriction map $\Omega^*(X) \to \Omega^*(A)$ is surjective for every inclusion $A\to X$ of simplicial sets. The proof of contractibility/extendability given here is taken from \cite[Proposition 7.13]{B-rathom}, which in turn is based on \cite[Lemma 3.2]{Getzler}.
\end{remark}

\section{Some simplicial homotopy theory} \label{sec:simplicial homotopy theory}
In this section we review some basic simplicial homotopy theory, organized in a way suitable for our applications. First let us recall the following definitions.
\begin{definition}
\begin{enumerate}
\item A surjective map of simplicial sets $p\colon E\to B$ is called a \emph{fibre bundle with fibre $F$} if for every simplex $\sigma\colon \Delta^n \to B$ the left vertical map in the pullback
$$
\xymatrix{\sigma^*(E) \ar[d] \ar[r] & E \ar[d]^-p \\ \Delta^n \ar[r]^-\sigma & B}
$$
is isomorphic to $F\times \Delta^n \to \Delta^n$ as a simplicial set over $\Delta^n$.

\item Let $G$ be a simplicial group. A \emph{principal $G$-fibration}, or \emph{principal $G$-bundle}, is a simplicial map $p\colon E\to B$ in which $E$ has a free $G$-action such that $p$ induces an isomorphism $E/G\cong B$. 
\end{enumerate}
\end{definition}

The following is standard (see e.g.~ \cite{May} or \cite{GJ}).

\begin{proposition}
\begin{enumerate}
\item Every simplicial group $G$ is a Kan complex.

\item Every principal $G$-bundle is a fibre bundle with fibre $G$.

\item Every fibre bundle with fibre a Kan complex is a Kan fibration. \hfill $\square$
\end{enumerate} 
\end{proposition}

Let $G$ be a simplicial group acting on a simplicial set $X$. For a vertex $v\in X_0$, the orbit $Gv$ is the subspace of $X$ which in level $n$ is the orbit of the degenerate $n$-simplex $s_0^n(v)$ under the action of $G_n$. Similarly, the stabilizer $G_v$ is the simplicial subgroup of $G$ which in level $n$ is the stabilizer of $s_0^n(v)$.

\begin{proposition} \label{prop:action}
Let $G$ be a simplicial group acting on a simplicial set $X$.
\begin{enumerate}
\item For every vertex $v\in X_0$, the orbit $Gv$ is a Kan complex and the map
$$G\to Gv$$
is a principal $G_v$-fibration. If $G$ is contractible, then this is a universal principal $G_v$-fibration, whence $Gv\simeq BG_v$.

\item The canonical map
$$\bigsqcup_{[v]\in X_0/G_0} Gv \to X$$
is an isomorphism if and only if the orbit space $X/G$ is discrete. If this holds, then $X$ is a Kan complex. If in addition $G$ is connected, then $Gv$ coincides with the path component $X_v$ containing $v\in X_0$ and in particular
$$\pi_0(X) = X_0/G_0.$$

\item The canonical map
$$G\times_{G_0} X_0 \to X$$
is an isomorphism if and only if the orbit space $X/G$ is discrete and the stabilizer $G_v$ is discrete for every $v\in X_0$. If this holds and moreover $G$ is contractible, then $X$ is weakly equivalent to the nerve of the groupoid associated to the action of $G_0$ on $X_0$.
\end{enumerate}
\end{proposition}

\begin{proof}
That $G\to Gv$ is a principal $G_v$-fibration is immediate as $G_v$ acts freely on $G$ and $G/G_v\cong Gv$. That $Gv$ is a Kan complex follows from Lemma \ref{lemma:cancellation} below applied to $G\to Gv \to *$.

That $X/G$ is discrete means that $s_0^n\colon X_0/G_0 \to X_n/G_n$ is an isomorphism for every $n$. In other words, for each $x\in X_n$ there is a unique $[v]\in X_0/G_0$ such that $[x] = s_0^n[v]$, i.e., $x\in G_ns_0^n(v)$. This proves the first part of the second statement. The statement about $X$ being Kan follows since disjoint unions of Kan complexes are Kan. If $G$ is connected, then so is $Gv$, and this implies the second part.

For the third statement, note that $X/G$ being discrete is an obvious necessary condition. Assuming $X/G$ discrete, the decompositions of $X_0$ and $X$ as coproducts over all vertices $v$ of the orbits $G_0 v$ and $Gv$, respectively, reduce the statement to $G\times_{G_0} G_0v \to Gv$ being an isomorphism. Since $G_0v \cong G_0/(G_0)_v$ and $Gv \cong G/G_v$ this is equivalent to the statement that $(G_0)_v \to G_v$ is an isomorphism, which is equivalent to $G_v$ being discrete. If $G$ is contractible, then $G\times_{G_0} X_0$ is a model for the homotopy orbit space $X_0\horb G_0$. Another model for the homotopy orbit space is the bar construction $B\big(*,G_0,X_0\big)$, which is isomorphic to the nerve of the groupoid associated to the action of $G_0$ on $X_0$.
\end{proof}

\begin{proposition} \label{prop:kan fibration}
Let $G$ and $G'$ be simplicial groups acting on simplicial sets $X$ and $X'$, respectively, and assume that the orbit spaces $X/G$ and $X'/G'$ are discrete. Let $\varphi\colon G \to G'$ be a map of simplicial groups and $f\colon X\to X'$ a map of $G$-spaces. If $\varphi$ is a Kan fibration, then so is $f$.
\end{proposition}

\begin{proof}
It follows from the decompositions of $X$ and $X'$ as coproducts of orbits (Proposition \ref{prop:action}(2)) that it is enough to verify that $Gv \to G'f(v)$ is a Kan fibration for every $v\in X_0$. By Proposition \ref{prop:action}(1), the vertical maps in the diagram
$$\xymatrix{G \ar[r]^-\varphi \ar[d] & G' \ar[d] \\ Gv \ar[r] & G'f(v).}$$
are surjective Kan fibrations. The lemma below then implies that the bottom horizontal map is a Kan fibration if $\varphi$ is.
\end{proof}

\begin{lemma} \label{lemma:cancellation}
If $X\xrightarrow{p} Y \xrightarrow{q} Z$ are simplicial maps such that $p$ and $qp$ are Kan fibrations and $p$ is surjective, then $q$ is a Kan fibration.
\end{lemma}

\begin{proof}
This is \cite[Exercise V.3.8]{GJ}.
\end{proof}

\section{Exponential groups and Maurer--Cartan spaces}
If $L$ is a complete dg Lie algebra, then we can form the simplicial complete dg Lie algebra $\Omega_\bullet(L) = \Omega_\bullet \ctensor L$. The \emph{exponential group} of $L$ is the simplicial group
$$\exp_\bullet(L) = \exp\big(Z^0\Omega_\bullet(L)\big)$$
obtained by applying $\exp(-)$ levelwise to the simplicial complete Lie algebra $Z^0\Omega_\bullet(L)$. The \emph{nerve} or \emph{Maurer--Cartan space} is the simplicial set
$$\MC_\bullet(L) = \MC\big(\Omega_\bullet(L)\big).$$
The levelwise gauge action defines an action of the simplicial group
$$\gauge_\bullet(L) = \exp\big(\Omega_\bullet^0(L)\big)$$
on $\MC_\bullet(L)$. We will now apply the results of the previous section to this action in order to deduce a number of basic properties of the Maurer--Cartan space.

\begin{lemma} \label{lemma:main lemma}
Let $L$ be a complete dg Lie algebra.
\begin{enumerate}
\item The simplicial group $\gauge_\bullet(L)$ is contractible.
\item The orbit space $\MC_\bullet(L)/\gauge_\bullet(L)$ is discrete and isomorphic to $\MCQ(L)$.

\item The stabilizer of $\tau\in\MC(L)$ under the action of $\gauge_\bullet(L)$ on $\MC_\bullet(L)$ is the simplicial group $\exp_\bullet(L_\tau)$. It is discrete if and only if $L=L^{\geq 0}$.
\end{enumerate}
\end{lemma}

\begin{proof}
The underlying simplicial set of $\gauge_\bullet(L)$ is $\Omega_\bullet^0(L)$, the contractibility of which follows from Lemma \ref{lemma:contractibility}. Towards the second claim, note that the orbit space is the simplicial set $\MCQ(\Omega_\bullet(L))$ obtained by applying $\MCQ(-)$ levelwise. The map $\eta_\bullet\colon L\to \Omega_\bullet(L)$ induced by the unit of $\Omega_\bullet$ induces a simplicial map from the constant simplicial set $\MCQ(L)$ to $\MCQ(\Omega_\bullet(L))$. The map $\epsilon_n = \partial_0^n\tensor 1 \colon \Omega_n(L)\to L$ satisfies $\epsilon_n \eta_n = 1$. In particular, $\MCQ(\epsilon_n)$ is surjective. Lemma \ref{lemma:GM light} applies to $\epsilon_n$, because if $I=\ker(\epsilon_n)$, then $F^kI/F^{k+1}I$ may be identified with $\overline{\Omega}_n\tensor (F^kL/F^{k+1}L)$, and $\overline{\Omega}_n = \ker(\partial_0^n)$ has zero cohomology by Lemma \ref{lemma:poincare lemma}, so $\MCQ(\epsilon_n)$ is a bijection. It follows that $\MCQ(\eta_\bullet)\colon \MCQ(L) \to \MCQ(\Omega_\bullet(L))$ is a levelwise bijection. The third claim follows from Lemma \ref{lemma:stabilizer} and Proposition \ref{prop:abelian}(4) below.
\end{proof}

\begin{theorem} \label{thm:mc}
Let $L$ be a complete dg Lie algebra.
\begin{enumerate}

\item The canonical map
$$\bigsqcup_{[\tau]\in \MCQ(L)} \gauge_\bullet(L)\tau \to \MC_\bullet(L)$$
is an isomorphism. Hence, $\MC_\bullet(L)$ is a Kan complex and the orbit $\gauge_\bullet(L)\tau$ agrees with the path component $\MC_\bullet(L)_\tau$ containing $\tau$. In particular, there is a natural bijection
$$\pi_0\MC_\bullet(L) \cong \MCQ(L),$$
i.e., two Maurer--Cartan elements are gauge equivalent if and only if they belong to the same path component of the Maurer--Cartan space.

\item For every $\tau\in \MC(L)$, the map given by acting on $\tau$,
$$\gauge_\bullet(L) \to \MC_\bullet(L)_\tau,$$
is a universal principal $\exp_\bullet(L_\tau)$-bundle. In particular, $\MC_\bullet(L)_\tau$ is weakly equivalent to the classifying space $B\exp_\bullet(L_\tau)$. Equivalently, $\exp_\bullet(L_\tau)$ is a simplicial group model for the loop space $\Omega \MC_\bullet(L)_\tau$ based at $\tau$.

\item The canonical map
$$\gauge_\bullet(L) \times_{\gauge(L)} \MC(L) \to \MC_\bullet(L)$$
is an isomorphism if and only if $L = L^{\geq 0}$. If this holds, then $\MC_\bullet(L)$ is weakly equivalent to the nerve of the Deligne groupoid, i.e., the groupoid associated to the action of $\gauge(L)$ on $\MC(L)$.
\end{enumerate}
\end{theorem}

\begin{proof}
This is immediate from Lemma \ref{lemma:main lemma} and Proposition \ref{prop:action}.
\end{proof}

\begin{remark}
The fact that $\MC_\bullet(L)$ is a delooping of $\exp_\bullet(L)$ is shown for positively graded $L$ of finite type in \cite[Corollary 3.10]{B-fb}, but the proof relies on results of \cite[Chapter 25]{FHT-RHT} that are not applicable if $L$ is unbounded or not of finite type. A version of the statement is embedded in \cite[\S4.7]{Pridham}, but a direct elementary proof has been missing from the literature as far as we know.
\end{remark}

\begin{theorem} \label{thm:kan fibration}
If $f\colon L\to L'$ is a surjective morphism of filtered complete dg Lie algebras, then the induced map $\MC_\bullet(L) \to \MC_\bullet(L')$ is a Kan fibration.
\end{theorem}

\begin{proof}
Apply Proposition \ref{prop:kan fibration} to the induced maps $\gauge_\bullet(f)\colon \gauge_\bullet(L)\to \gauge_\bullet(L')$ and $\MC_\bullet(f)\colon \MC_\bullet(L) \to \MC_\bullet(L')$. Forgetting group structures, the map $\gauge_\bullet(f)$ is equal to $\Omega_\bullet^0(L)\to \Omega_\bullet^0(L')$, which also underlies a surjective map of simplicial abelian groups and is hence a Kan fibration (see e.g.~\cite[Lemma III.2.10]{GJ}).
\end{proof}

\begin{remark}
Let us briefly comment on the relation to existing proofs that $\MC_\bullet(-)$ takes surjections to Kan fibrations, especially \cite[\S2.2]{Hinich}, \cite[\S8]{Hinich-coalg}, \cite[\S4]{Getzler}, and \cite[\S11.3]{BFMT}. The paper \cite{Hinich} is mainly concerned with the case $L=L^{\geq 0}$ and it should be noted that the decomposition \cite[2.2.2 Proposition]{Hinich} is only valid in this case by Theorem \ref{thm:mc}(3). The proofs in \cite[\S4]{Getzler} and \cite[\S11.3]{BFMT} use certain explicit homotopies on $\Omega_\bullet$ to decompose $\MC_n(L)$ as $\MC(L)$ times an auxiliary set (cf.~\cite[Lemma 4.6]{Getzler} and \cite[Lemma 11.11]{BFMT}). These decompositions should be compared to the decomposition in Theorem \ref{thm:mc}(1), which is less refined but sufficient for proving the Kan property. Unlike Getzler's proof \cite{Getzler}, the proof given here does not apply to $L_\infty$-algebras $L$, because the groups $\gauge_\bullet(L)$ and $\exp_\bullet(L)$ are not available. The suggestion \cite[Note 8.2.7]{Hinich-coalg} is close in spirit to the proof given here.
\end{remark}

\section{Homotopy groups}
In this section, we compute the homotopy groups of the exponential group and the Maurer--Cartan space. Since the homotopy groups of $\exp_\bullet(L)$ are the same as those of the underlying simplicial set $Z^0\Omega_\bullet(L)$, the computation for $\exp_\bullet(L)$ reduces to the case of abelian $L$. The homotopy groups of $\MC_\bullet(L)$ can then be computed using the fact that $\Omega\MC_\bullet(L)_\tau$ is weakly equivalent to $\exp_\bullet(L_\tau)$ by Theorem \ref{thm:mc}(2).

Consider the differential form $\vol{n} \in \Omega_n^n$ defined by
\begin{equation*}
\vol{n}  = n! dt_1\ldots dt_n.
\end{equation*}
It satisfies $d\vol{n} = \partial\vol{n} =0$ and $\int_{\Delta^n} \vol{n} = 1$. Here and below, $\partial$ denotes the alternating sum $\sum_i (-1)^i \partial_i$ of the face maps.

\begin{proposition} \label{prop:abelian}
Let $V$ be a complete filtered cochain complex, consider the simplicial cochain complex $\Omega_\bullet(V) = \Omega_\bullet \ctensor V$, and let $k$ be an integer.
\begin{enumerate}
\item The simplicial vector space $\Omega_\bullet^k(V)$ is contractible.
\item The simplicial vector space $H^k\big(\Omega_\bullet(V)\big)$ is discrete and isomorphic to $H^k(V)$.
\item The map
\begin{gather*}
f\colon H_k(V) \to \pi_k\big(Z^0\Omega_\bullet(V)\big), \\
[x] \mapsto [\vol{k}\tensor x],
\end{gather*}
is a natural isomorphism.
\item The simplicial vector space $Z^0\Omega_\bullet(V)$ is discrete if and only if $V=V^{\geq 0}$.
\end{enumerate}
\end{proposition}

\begin{proof}
The first two claims are direct consequences of Lemma \ref{lemma:poincare lemma} and Lemma \ref{lemma:contractibility}. Towards the third claim, let $\Omega^k$, $Z^k$, $B^k$, and $H^k$, denote the simplicial vector spaces of $k$-dimensional cochains, cocycles, coboundaries, and cohomology, respectively, of $\Omega_\bullet(V)$. The first statement together with the long exact sequence of homotopy groups associated to the short exact sequence of simplicial vector spaces
$$0 \to Z^{k-1} \to \Omega^{k-1} \to B^k \to 0$$
show that $\pi_\ell(B^k) \cong \pi_{\ell-1}(Z^{k-1})$ for all integers $k,\ell$. In particular, $\pi_0(B^{k}) = 0$. Similarly, using the second statement and the sequence
$$0\to B^k \to Z^k \to H^k \to 0,$$
one sees that $\pi_\ell(B^k) \cong \pi_\ell(Z^k)$ for all $\ell > 0$ and $\pi_0(Z^k) \cong \pi_0(H^k) = H^k(V)$ for all $k$. Combining the above facts yields natural isomorphisms
$$\pi_k(Z^0) \cong \pi_{k-1}(Z^{-1}) \cong \ldots \cong \pi_0(Z^{-k}) \cong H^{-k}(V).$$
Explicitly, the isomorphism $\pi_\ell(Z^k) \to \pi_{\ell-1}(Z^{k-1})$ for $\ell>0$ sends an element of the form $[dz]$ to $[\partial z]$. Using this and the equations $\vol{k} = d\volo{k-1}$ and $\partial \volo{k-1} = \vol{k-1}$, where $\volo{k-1}\in \Omega_{k}^{k-1}$ denotes the differential form
$$\volo{k-1} = (k-1)!\sum_{i=1}^{k} (-1)^{i-1} t_i dt_1\ldots \widehat{dt_i} \ldots dt_{k},$$
one sees that the class $[\omega^k \tensor x] \in \pi_k(Z^0)$ corresponds to $[x] \in H^{-k}(V)$.

Since $H^0$ is discrete, $Z^0$ is discrete if and only if $B^0$ is. If $x$ is a non-zero element of $V^{-k}$ for $k>0$, then $d(t_0dt_1\ldots dt_{k-1}\tensor x)$ represents a non-zero element of the cokernel of $B^0V \to B^0\Omega_{k-1}(V)$, showing $B^0$ is not discrete. Conversely, if $V=V^{\geq 0}$, then $\Omega = \Omega^{\geq 0}$, whence $B^0$ is zero and in particular discrete.
\end{proof}

\begin{theorem} \label{thm:homotopy of exp}
There is a natural isomorphism of groups
\begin{equation} \label{eq:pi_0}
\exp \big(H_0(L)\big) \to \pi_0\big(\exp_\bullet(L)\big)
\end{equation}
and for every $k\geq 1$ a natural isomorphism of abelian groups
$$H_k(L) \to \pi_k\big(\exp_\bullet(L)\big)$$
compatible with Samelson products and the actions of the groups in \eqref{eq:pi_0}. In both cases, the map is given by
$$[x] \mapsto \big[\omega^k \tensor x\big]$$
for cycles $x\in L_k$.
\end{theorem}

\begin{proof}
The underlying simplicial set of $\exp_\bullet(L)$ is $Z^0\Omega_\bullet(L)$ so Proposition \ref{prop:abelian} yields a natural bijection
$H_k(L) \to \pi_k\big(\exp_\bullet(L) \big)$ for all $k\geq 0$, which is an isomorphism of abelian groups for $k>0$. It remains to identify the group structure on $\pi_0(\exp_\bullet(L))$ and its action on the higher homotopy groups that come from the simplicial group structure on $\exp_\bullet(L)$. To do this, one can argue using the zig-zag of $H_0$-isomorphisms
\begin{equation} \label{eq:H_0}
L \leftarrow L\langle 0\rangle \rightarrow H_0(L).
\end{equation}
Here $L\langle k\rangle \subseteq L$ is defined by $L\langle k\rangle_i = L_i$ for $i>k$, $L\langle k \rangle_k = Z_kL$ and $L\langle 0\rangle_i = 0$ for $i<k$. The induced maps on $\pi_0(\exp_\bullet(-))$ are group homomorphisms and, by what we have just seen, bijections. Since $\exp_\bullet(H_0(L))$ is isomorphic to the discrete simplicial group $\exp (H_0(L))$, the claim about the group structure follows. The claim about the action on higher homotopy groups can be verified similarly by considering the zig-zag of $H_k$-isomorphisms
$$L \leftarrow L\langle k\rangle \rightarrow H_k(L)$$
with the actions of the $0$-cocycles of the dg Lie algebras in \eqref{eq:H_0}. The statement about Samelson products is proved in Corollary \ref{cor:samelson} below.
\end{proof}

\begin{corollary}
For every $\tau\in \MC(L)$ there is a natural isomorphism of groups
\begin{equation} \label{eq:groups}
\exp\big(H_0(L_\tau)\big)\to \pi_1\big(\MC_\bullet(L),\tau\big),
\end{equation}
and for every $k\geq 1$ a natural isomorphism of abelian groups
$$H_k(L_\tau) \to \pi_{k+1}\big(\MC_\bullet(L),\tau\big),$$
compatible with the actions of the groups in \eqref{eq:groups} and with Whitehead products up to a sign. In both cases, the map is given by
$$[x] \mapsto \big[\tau\tensor 1 - \omega^{k+1}\tensor x\big]$$
for $x\in L_k$ such that $d_\tau(x)= 0$.
\end{corollary}

\begin{proof}
This follows from Theorem \ref{thm:homotopy of exp} and Theorem \ref{thm:mc}(2) together with the fact that the canonical isomorphism $\pi_{*+1}(X) \to \pi_*(\Omega X)$ takes Whitehead products to Samelson products up to a sign (see \cite[p.197]{Curtis}) and is compatible with the relevant group actions. To get the explicit formula, one can compute the connecting homomorphism $\partial \colon \pi_{k+1}(\MC_\bullet(L),\tau) \to \pi_k(\exp_\bullet(L_\tau))$ of the fibration $\exp_\bullet(L_\tau) \to \gauge_\bullet(L) \to \MC_\bullet(L)_\tau$ using the recipe in \cite[p.28]{GJ}. The key observation is that the gauge action of $\volo{k}\tensor x$ on $\tau\tensor 1$ is equal to $\tau\tensor 1 - \vol{k+1}\tensor x$. This implies that the connecting homomorphism sends $[\tau\tensor 1 - \omega^{k+1}\tensor x]$ to $[\omega^k \tensor x]$.
\end{proof}

\begin{remark}
This simplifies the computation of the homotopy groups of $\MC_\bullet(L)$ given in \cite{B} in the case when $L$ is a dg Lie algebra. It also enhances the computation by showing that the isomorphism is compatible with Whitehead products (up to a sign) and the action of the fundamental group.
\end{remark}

\section{A Dold--Kan correspondence for Lie algebras}
The normalized chains $N$ is a (lax) symmetric monoidal functor from simplicial abelian groups to chain complexes. In particular, 
if $L_\bullet$ is a simplicial Lie algebra, then $NL_\bullet$ is a dg Lie algebra. Explicitly, the Lie bracket $[x,y]\in NL_{p+q}$ of $x\in NL_p$ and $y\in NL_q$ is given by
\begin{equation} \label{eq:Lie bracket}
[x,y] = \sum_{(\mu,\nu)} \sgn(\mu,\nu) \big[s_\nu x,s_\mu y\big].
\end{equation}
The sum is over all $(p,q)$-shuffles $(\mu,\nu) = (\mu_1,\ldots,\mu_p,\nu_1,\ldots,\nu_q)$, i.e., permutations of $0,1,\ldots,p+q-1$ that satisfy $\mu_1< \ldots < \mu_p$ and $\nu_1 < \ldots < \nu_q$, and $s_\mu = s_{\mu_p} \ldots s_{\mu_1}$ and $s_\nu = s_{\nu_q} \ldots s_{\nu_1}$, and $\sgn(\mu,\nu) = \pm 1$ is the sign of the shuffle, cf.~e.g.~\cite[p.243]{Mac Lane}. The inverse functor $\Gamma$ from chain complexes to simplicial abelian groups featured in the Dold--Kan correspondence (cf.~\cite[Corollary III.2.3]{GJ}) is not symmetric monoidal\footnote{but it is $E_\infty$-monoidal \cite{Richter}} so the simplicial abelian group $\Gamma L$ does not in general admit a Lie algebra structure when $L$ is a dg Lie algebra. The construction $Z^0\Omega_\bullet(L)$ could be viewed as a fattened up model for $\Gamma L$ that admits a Lie algebra structure.

\begin{theorem} \label{thm:Lie Dold Kan}
Let $L$ be a complete dg Lie algebra. The map
\begin{gather*}
I\colon \Omega_n^n\tensor L_n  \to L_n, \quad 
I(\omega \tensor x)  = (-1)^{\binom{n}{2}}\int_{\Delta^n} \omega \cdot x,
\end{gather*}
induces a natural quasi-isomorphism of dg Lie algebras
$$I\colon NL_\bullet  \to L,$$
where $L_\bullet$ denotes the simplicial Lie algebra $Z^0\Omega_\bullet(L)$.
\end{theorem}

\begin{proof}
Elements of $\Omega_n^0(L)$ are of the form $\xi = \sum_{i=0}^n \omega_i\tensor x_i$, where $\omega_i \in \Omega_n^i$ and $x_i\in L_i$. If $\xi$ is a cycle, then a look at the component of $d\xi=0$ in $\Omega_n^n\tensor L_{n-1}$ shows
$$(-1)^n \omega_n \tensor dx_n + d\omega_{n-1}\tensor x_{n-1} = 0,$$
which implies
$$(-1)^n \int_{\Delta^n} \omega_n \cdot dx_n + \int_{\Delta_n} d\omega_{n-1} \cdot x_{n-1} = 0.$$
Since $\int_{\Delta^n} d\omega_{n-1} = \int_{\Delta^{n-1}} \partial\omega_{n-1}$ by Stokes' theorem, this shows that $I(\partial \xi) = dI(\xi)$, so the restriction of $I$ to $N Z^0\Omega_\bullet(L)$ is a chain map.
The verification that $I$ respects Lie brackets uses the formula
$$\int_{\Delta^p} \alpha \cdot \int_{\Delta^q} \beta = \sum_{(\mu,\nu)} \sgn(\mu,\nu) \int_{\Delta^{p+q}} s_\nu \alpha \cdot s_\mu \beta$$
which holds as both sides compute the integral of the differential $(p+q)$-form $\alpha \times \beta$ over $\Delta^p \times \Delta^q$, on one hand using the Fubini theorem and on the other hand using the standard decomposition of $\Delta^p \times \Delta^q$ into a union of $(p+q)$-simplices (cf.~e.g.~\cite[pp.~243--244]{Mac Lane}).
Using $\int_{\Delta^n} \omega^n = 1$, one sees that the map induced by $I$ in homology is left inverse to the isomorphism $f$ in Proposition \ref{prop:abelian}(3). This implies that $I$ is a quasi-isomorphism and, as a by-product, that $f$ respects Lie brackets.
\end{proof}

\begin{corollary} \label{cor:samelson}
Let $L$ be a complete dg Lie algebra. The isomorphism
$$f\colon H_k(L) \to \pi_k(\exp_\bullet(L)),\quad f[x] = [\omega^k \tensor x],$$
takes Lie brackets to Samelson products.
\end{corollary}

\begin{proof}
By Curtis' formula for Samelson products \cite[p.197]{Curtis}, we have that
\begin{equation} \label{eq:curtis}
\big\langle\omega^p \tensor x,\omega^q \tensor y\big\rangle = \prod_{(\mu,\nu)} \big[s_\nu\omega^p\tensor x, s_\mu\omega^q \tensor y\big ]_{BCH}^{\sgn(\mu,\nu)},
\end{equation}
where the product over the $(p,q)$-shuffles $(\mu,\nu)$ is taken in a certain order that will turn out not to matter for us. The commutator with respect to the BCH product satisfies
$$[a,b]_{BCH} = [a,b] + \mbox{(higher terms)}.$$
All products of length $>2$ formed out of $s_\nu\omega^p$ and $s_\mu\omega^q$ are zero since $\Omega_{p+q}^k = 0$ for $k>p+q$, so it follows that the higher terms vanish when $a= s_\nu\omega^p \tensor x$ and $b = s_\mu\omega^q \tensor y$. Similarly, since the BCH product $*$ satisfies
$$a* b = a+b + \mbox{(higher terms)},$$
the product in \eqref{eq:curtis} reduces to a sum. This shows that the Samelson product on the homotopy groups of $\exp_\bullet(L)$ agrees with the Lie bracket \eqref{eq:Lie bracket} under the identification $\pi_k(\exp_\bullet(L)) = H_k(NL_\bullet)$. Since $f$ is inverse to $H_*(I)$ and the latter preserves Lie brackets by the previous proposition, the result follows.
\end{proof}

\section{Relation to Quillen's functor}
Quillen \cite{Quillen} defined a functor
$$\lambda\colon \Top_2 \to \DGL_1$$
from the category of simply connected pointed spaces to the category of positively graded dg Lie algebras over $\QQ$ and showed it induces an equivalence after localizing with respect to the rational homotopy equivalences and the quasi-isomorphisms, respectively. If $L$ is a positively graded dg Lie algebra over $\QQ$, then the geometric realization $|\MC_\bullet(L)|$ is simply connected and it is natural to ask whether $\lambda |\MC_\bullet(L)|$ is quasi-isomorphic to $L$. For $L$ of finite type, this statement is easily seen to be equivalent to a conjecture formulated by Baues--Lemaire in 1977 \cite[Conjecture 3.5]{BL}, which was proved by Majewski in 2000 \cite{Majewski}. F\'elix--Fuentes--Murillo \cite{FFM} recently gave a proof that does not require $L$ to be of finite type. However, this proof is indirect as it relies on identifying $\MC_\bullet(L)$ with the realization $\langle L \rangle$ of \cite[\S7]{BFMT} up to homotopy. A corollary of Theorem \ref{thm:mc} and Theorem \ref{thm:Lie Dold Kan} is the following more direct proof.

\begin{theorem}
Let $L$ be a positively graded dg Lie algebra over $\QQ$. There is a natural zig-zag of quasi-isomorphisms of dg Lie algebras
$$\lambda|\MC_\bullet(L)| \leftarrow \cdots \rightarrow L.$$
\end{theorem}

\begin{proof}
We adopt Quillen's setup \cite[p.211]{Quillen}. The simplicial set $\MC_\bullet(L)$ is $2$-reduced as $L$ is positively graded, so the unit of the adjunction between $\Top_2$ and the category of $2$-reduced simplicial sets induces a natural quasi-isomorphism 
\begin{equation} \label{eq:lambda}
N\prim \cQQ\Kangrp \MC_\bullet(L) \to \lambda |\MC_\bullet(L)|.
\end{equation}
Our goal is now to construct a natural quasi-isomorphism from the source of \eqref{eq:lambda} to $L$. Let $L_\bullet$ denote the simplicial nilpotent Lie algebra $Z^0\Omega_\bullet(L)$. Theorem \ref{thm:mc}(2) implies that there is a natural weak equivalence of reduced simplicial groups
\begin{equation} \label{eq:loop group}
\Kangrp \MC_\bullet(L) \to \exp(L_\bullet),
\end{equation}
where $\Kangrp$ denotes the Kan loop group (cf.~e.g.~\cite[\S V.5]{GJ}). 
Since $L_\bullet$ is levelwise nilpotent, Corollary 3.9 and (2.7) in Appendix A of \cite{Quillen} give natural isomorphisms
$$L_\bullet \to \prim \cU L_\bullet \xrightarrow{exp} \glike \cU L_\bullet,$$
where $\prim \cU$ and $\glike \cU$ denote primitives and group-like elements, respectively, in the completed universal enveloping algebra. Hence, we may identify the simplicial group $\exp(L_\bullet)$ with $\glike \cU L_\bullet$ and regard the weak equivalence \eqref{eq:loop group} as a map
\begin{equation} \label{eq:loop group rewritten}
\Kangrp \MC_\bullet(L) \to \glike \cU L_\bullet.
\end{equation}
Now, $\glike$ is the right adjoint in a Quillen equivalence between reduced simplicial groups and reduced simplicial complete Hopf algebras (\cite[Theorem II.4.8]{Quillen}). Since all objects in the latter category are fibrant and since Kan loop groups are always cofibrant, this implies that the adjoint of \eqref{eq:loop group rewritten},
$$\cQQ\Kangrp \MC_\bullet(L) \to \cU L_\bullet,$$
is a weak equivalence. This means that the induced map on primitives,
$$\prim \cQQ\Kangrp \MC_\bullet(L) \to \prim\cU L_\bullet \cong L_\bullet,$$
is a weak equivalence (cf.~\cite[Theorem II.4.7]{Quillen}). Applying normalized chains, we obtain natural quasi-isomorphisms
$$N\prim \cQQ\Kangrp \MC_\bullet(L) \to NL_\bullet \to L$$
where the last map is the quasi-isomorphism from Theorem \ref{thm:Lie Dold Kan}.
\end{proof}

\section{Quillen's universal dg coalgebra bundle and the gauge action} \label{sec:quillen}
We end this note with a remark about the relation to Quillen's theory of dg coalgebra bundles \cite[Appendix B \S5--\S6]{Quillen}. The gauge action does not appear explicitly in Quillen's work, but it is implicit in a sense that we will now explain. On p.291, Quillen defines the dg Lie algebra $s L \# L$, whose underlying graded vector space is $sL \oplus L$, where $sL$ is the suspension of $L$ and where the bracket and differential are determined by the equations
$$[sx,sy] = 0,\quad [sx,y] = s[x,y],\quad dsx = x-sdx,$$
and the requirement that $L$ is a dg Lie subalgebra. He then constructs a universal principal dg coalgebra bundle
$$UL \to U(sL\# L) \to \CC L,$$
where $\CC L$ is defined to be $U(sL\# L) \tensor_{UL} \kk$ (and then shown to be isomorphic to the standard complex for computing Lie algebra homology). Observing that $\exp_\bullet(L)$ and $\MC_\bullet(L)$ are isomorphic to the coalgebra realizations $\langle UL \rangle$ and $\langle \CC L\rangle$, respectively (cf.~\cite[\S3.3]{B-fb}), and that $Z^0(sL \# L)$ is isomorphic to $L^0$, one can rediscover the gauge action as the action of the simplicial group
$$\gauge_\bullet(L) = \exp_\bullet(sL\# L) = \langle U(sL\# L) \rangle$$
on the simplicial set
$$\MC_\bullet(L) = \langle \CC L \rangle = \langle U(sL \# L) \tensor_{UL} \kk \rangle$$
induced by the left action of the Hopf algebra $U(sL\# L)$ on $U(sL \# L) \tensor_{UL} \kk$.
In fact, the coalgebra realization of Quillen's universal principal $L$-bundle,
$$\langle UL \rangle \to \langle U(sL\# L) \rangle \to \langle \CC L \rangle,$$
is isomorphic to the universal principal $\exp_\bullet(L)$-bundle
$$\exp_\bullet(L) \to \gauge_\bullet(L) \to \MC_\bullet(L)$$
of Theorem \ref{thm:mc}(2). If $L$ is positively graded and of finite type, this is also isomorphic to the principal fibre bundle constructed in \cite[\S25(b)]{FHT-RHT}, cf.~\cite[\S3.4]{B-fb}.

\subsection*{Acknowledgements}
We thank Thomas Willwacher for prompting the writing of this note and for stimulating discussions.
The author was supported by the Swedish Research Council through grant no.~2021-03946.

\end{document}